\documentclass{birkjour}
\usepackage[active]{srcltx}
\usepackage{hyperref}
 \usepackage{color}
\usepackage{amsmath}
\usepackage{amsthm}
\usepackage{amssymb}

\renewcommand{\theequation}{\arabic{section}.\arabic{equation}}
\theoremstyle{plain}
\numberwithin{equation}{section}

\newcommand{\leaveout}[1]{}

\makeatletter
\renewcommand{\p@enumii}{}
\makeatother

\newcommand{\cB}{{\mathcal B}}

\newcommand{\cL}{{\mathcal L}}
\newcommand{\cM}{{\mathcal M}}
\newcommand{\cN}{{\mathcal N}}

\newcommand{\cK}{{\mathcal K}}

\newcommand{\cU}{{\mathcal U}}
\newcommand{\cX}{{\mathcal X}}
\newcommand{\cY}{{\mathcal Y}}

\newcommand{\krein}{Kre\u\i n}
\newcommand{\ud}{\,{\mathrm d}}

\newcommand\C{{\mathbb C}}
\newcommand\D{{\mathbb D}}
\newcommand\E{{\mathbb E}}

\newcommand\Z{{\mathbb Z}}

\newcommand\zplus{{\Z^{+}}}

\newcommand\T{{\mathbb T}}

\newcommand\zero{\set{0}}

\newcommand{\Xscr}{\mathcal X}

\newcommand{\Ipdp}[2]{\left\langle #1 , #2 \right\rangle}

\newcommand{\set}[1]{\left\lbrace #1 \right\rbrace}

\newcommand{\bi}{\begin{itemize}}
\newcommand{\ei}{\end{itemize}}
\newcommand{\be}{\begin{enumerate}}
\newcommand{\ee}{\end{enumerate}}

\newcommand{\sbm}[1]{\left[\begin{smallmatrix}#1
\end{smallmatrix}\right]}

\newcommand{\bbm}[1]{\begin{bmatrix}#1\end{bmatrix}}

\newcommand{\wtil}[1]{{\widetilde{#1}}}
\newcommand{\mat}[1]{\ensuremath{\begin{bmatrix} #1 \end{bmatrix}}}

\newcommand{\bx}{{\bf x}} 
\newcommand{\by}{{\bf y}} 

\newtheorem{lemma}{Lemma}[section]
\newtheorem{theorem}[lemma]{Theorem}

\theoremstyle{definition}

\begin{document}

\title[A new systems theory perspective on canonical Wiener-Hopf factorization]{A new systems theory perspective on canonical Wiener-Hopf factorization on the unit circle}

\author[S. ter Horst]{S. ter Horst}
\address{Department of Mathematics, Research Focus Area:\ Pure and Applied Analytics, North-West University, Potchefstroom, 2531 South Africa and DSI-NRF Centre of Excellence in Mathematical and Statistical Sciences (CoE-MaSS)}
\email{Sanne.TerHorst@nwu.ac.za}

\author[M. Kurula]{M. Kurula}
\address{{\AA}bo Akademi Mathematics, Henriksgatan 2, 20500 {\AA}bo, Finland}
\email{Mikael.Kurula@abo.fi}

\author[A.C.M. Ran]{A.C.M. Ran}
\address{Department of Mathematics, Faculty of Science, VU Amsterdam, De Boelelaan 1111, 1081 HV Amsterdam, The Netherlands and Research Focus: Pure and Applied Analytics, North-West University, Potchefstroom, South Africa}
\email{A.C.M.Ran@vu.nl}

\thanks{This work is based on the research supported in part by the National Research Foundation of South Africa (Grant Numbers 118513, 127364 and 145688), the DSI-NRF Centre of Excellence in Mathematical and Statistical Sciences (Grant Number 2024-030-OPA) and the Magnus Ehrnrooth Foundation.}


% 47A68(1980–now)Factorization theory (including Wiener-Hopf and spectral factorizations) of linear operators
% 47A63(1991–now)Linear operator inequalities
% 47A48(1991–now)Operator colligations (= nodes), vessels, linear systems, characteristic functions, realizations, etc.
% 47A56(1980–now)Functions whose values are linear operators (operator- and matrix-valued functions, etc., including analytic and meromorphic ones)
% 93B28(1980–now)Operator-theoretic methods [See also 47A48, 47A57, 47B35, 47N70]
% 93C05(1973–now)Linear systems in control theory
% 93D25(1980–now)Input-output approaches in control theory

\begin{abstract}
We establish left and right canonical factorizations of Hilbert-space operator-valued functions $G(z)$ that are analytic on neighborhoods of the complex unit circle $\T$ and the origin $0$ and that have the form $G(z)=I+F(z)$ with $F(z)$ taking strictly contractive values on $\T$. Such functions can be realized as transfer functions of infinite dimensional dichotomous discrete-time linear systems, and we employ the strict bounded real lemma for this class of operators, together with associated \krein\ space theory, to derive explicit formulas for the left and right canonical factorizations.  
\end{abstract}

\maketitle

{\it Dedicated to the memory of Heinz Langer, in admiration of his contributions to operator theory.}

%\tableofcontents

%%%%%%%%%%%%%%%%%%%%%%%%%%%%%%%%%%

\section{Introduction}\label{S:Intro}

Let $G(z)$ be an operator-valued function which takes bounded and invertible values on a closed rectifiable contour $\Gamma$, and assume that $G(z)$ is analytic on the contour; throughout this paper, ``operator'' will always mean bounded linear Hilbert space operator and when an operator is said to be invertible this implies that the inverse is bounded. We say that $G(z)$ admits a \emph{right canonical Wiener-Hopf factorization} on the contour 
when $G(z)= V_-(z)V_+(z)$, where $V_+(z)$ and $V_+(z)^{-1}$ extend analytically on the domain inside $\Gamma$, while $V_-(z)$ and $V_-(z)^{-1}$ extend analytically on the domain outside $\Gamma$. Likewise, we say that $G(z)$ admits \emph{left canonical Wiener-Hopf factorization} on the contour when $G(z)=W_+(z)W_-(z)$, where $W_+(z)$ and $W_-(z)$ have the same properties as $V_+(z)$ and $V_-(z)$, respectively. Such factorizations play an important role in establishing invertibility of Toeplitz, singular integral and Wiener-Hopf integral operators, and several other applications; See, e.g., \cite{BGKR08, BGKR10, BS, CGOT3, FKR10, GF, GGKbook90, GGKOT63}. An extensive review of the literature and an overview of results concerning these factorizations and their applications can be found in \cite{GKS}. 

In the seventies a series of papers by Gohberg and Leiterer appeared around the topic of canonical factorization. Most of these papers are in Russian; in English, most of the results contained in them can be found in \cite{GLOT192}. One of the main results is the following: Let $G(z)=I+F(z)$ be a Hilbert space operator-valued function that is analytic on a neighborhood of the unit circle $\T$, with $I$ indicating the identity operator and with $\max_{z\in\T} \|F(z)\| < 1$; since $F$ is analytic on a neighborhood of $\T$, the function $\|F(z)\|$ is continuous on the compact set $\T$ and hence the maximum indeed exists.   Then $G(z)$ admits both left and right canonical Wiener-Hopf factorization (cf., Theorem 8.5.3 and Corollary 8.5.4 in \cite{GLOT192}). A similar result holds for operator-valued functions with positive real part on the unit circle. In fact, these results hold true on any circular contour in the extended complex plane, that is, on circles and lines (which are interpreted as circles through infinity in the extended complex plane). What makes the result on the strictly contractive case so fascinating is that, in fact, this property characterizes the circular contours, in the following sense: If on a given Jordan curve $\gamma$, any rational matrix function $G(z)=I+F(z)$, with $\|F\|_{\infty,\gamma}<1$ on $\gamma$, allows both left and right canonical Wiener-Hopf factorization, then it was shown in \cite{MaMa76, ViMa73} that this curve must be a circle or a line; for this it in fact suffices to have both left and right canonical Wiener-Hopf factorization for $2 \times 2$ rational matrix functions.

Also in the seventies, the concept of \emph{realization} of a rational matrix function was developed in the control engineering literature. This concept was applied to factorization of matrix and operator valued functions in \cite{BGKOT1}, which later on was revised, expanded and updated in \cite{BGKR08, BGKR10}. In particular, in \cite{BGKR10} the focus is on canonical factorization of rational matrix valued functions and applications as diverse as the transport equation, optimal control and $H^\infty$-control. In 
Chapter 16 the bounded real lemma is discussed, see also, e.g. \cite{LaRoBook}. The bounded real lemma states that a rational matrix function in realized form takes contractive values on the unit circle or on the imaginary line if and only if a certain associated algebraic Riccati equation has a stabilizing solution. 
Using the solution to the algebraic Riccati equation, for a strictly proper rational matrix function $F(z)$ which has contractive values on the imaginary line, it is then shown that $G(z)=I+F(z)$ admits canonical Wiener-Hopf factorizations, both left and right. These results are based on \cite{GoRa93}. The analogue for rational matrix functions with a positive real part on the imaginary line was treated in \cite{RaRo92}, in the same manner.

The proof of the result of Gohberg and Leiterer (see \cite{GLOT192}) is based on fairly deep analytical tools. In contrast, the approach via realization theory using the bounded real lemma may be viewed as falling out of a combination of the main factorization approach of \cite{BGKOT1} with some theory of indefinite inner product spaces (see \cite{AzIobook,GheoBook22}). This is the approach taken in \cite{GoRa93}. In the present paper we shall return to the theorem of Gohberg and Leiterer, in an infinite dimensional Hilbert space setting, but starting from a realization approach. Instead of the Riccati equation from the bounded real lemma, we use the Kalman-Yakubovich-Popov inequality (abbreviated to KYP-inequality). This inequality and its use for properties of the transfer function and the corresponding system was discussed in great detail in the infinite dimensional case in a series of papers by several groups of authors, e.g. \cite{AKP06, AKP16} and \cite{BGtH18a, BGtH18b, BGtH18c} in discrete time, and for the continuous-time case in \cite{ArStKYP,BthK22}. Our starting point will be the transfer function of a dichotomous system as in \cite{BGtH18c}. This will enable us to use methods and results from the theory of \krein\ spaces, which can be found in \cite{AzIobook, GheoBook22, IKLbook}.
The theory of \krein\ spaces has been applied to many factorization problems for matrix and operator functions that have special properties. For rational matrix valued functions see, e.g., \cite{BGKR10}. Perhaps the first application of methods of indefinite inner product spaces to factorization problems for operator polynomials are to be found in work of M.G. \krein\ and H. Langer, see \cite{KL64, KL65}.

We observe that the class of functions considered in \cite{GLOT192} is slightly larger than the class of functions we consider. It turns out that the transfer function of a dichotomous system is a function which is analytic around the unit circle as well as at the origin, whose values are bounded linear operators between Hilbert spaces. In contrast, \cite{GLOT192} considers functions which may not be analytic at the origin. The methods we apply enable us to provide explicit formulas for the factors, in contrast to \cite{GLOT192}. 

In \cite{BGtH18c}, basic theory of Hilbert-space linear discrete time invariant \emph{dichotomous systems} of the form
\begin{equation}\label{eq:dichotomous}
	\left\{
		\begin{aligned}
			x(n+1) &= A x(n)+ Bu(n) \\	
			y(n) &= Cx(n)+Du(n),\quad n\in\zplus,
			\quad x(0)=x_0~\text{given},
		\end{aligned}
	\right.
\end{equation}
is discussed; here the word \emph{dichotomous} refers to the main operator $A$, more specifically to the fact that $A$ has no spectrum on the complex unit circle $\T$. The \emph{transfer function} of \eqref{eq:dichotomous} is
$$
	F(z)=D+zC(I-z A)^{-1} B,
$$
defined for all $z\in\C$, such that $I-zA$ is invertible: $z\in\rho(A)^{-1}\cup \{0\}$. We will, in particular, be interested in the case where
$$
	\|F\|_{\infty,\T}:=\max_{z\in\T} \|F(z)\|<1.
$$
Our main result, for which we shall present much more detailed versions in Section 3, is the following.

\begin{theorem}[main theorem, simplified]
Assume that the Hilbert space operator valued function $G(z)$ is of the following type, with $A$ dichotomous and $I+D$ invertible:
\begin{equation*}%\label{eq:Fspecial}
	 G(z)=I+F(z),\qquad
	F(z)=D+zC(I-z A)^{-1} B,\qquad
	\|F\|_{\infty,\T}<1.
\end{equation*}
Then $G(z)$ has a \emph{right canonical Wiener-Hopf factorization} $G(z)=V_-(z)\,V_+(z)$, i.e.,  $V_+(z)$ and $V_+(z)^{-1}$ are analytic on a neighborhood of the closed unit disc $\overline \D$, while $V_-(z)$ and $V_-(z)^{-1}$ are analytic on a neighborhood of the closed complement $\overline \E$ of $\D$. 
There is also a \emph{left} canonical Wiener-Hopf factorization $G(z)=W_+(z)\,W_-(z)$, where $W_\pm(z)$ have the same properties as $V_\pm(z)$.
\end{theorem}

The paper is organized as follows. In Section \ref{sec:BRL-Krein}, we discuss the class of infinite dimensional, discrete-time dichotomous systems, focusing on the case where the transfer function has strictly contractive values on the unit circle. The strict KYP inequality provides us with an invertible selfadjoint operator $H$, which we use as the Gram operator for an indefinite inner product on the state space. That turns the state space into a \krein\ space, and it turns out that in our case, the main operator $A$ is a bicontraction in this \krein\ space. Well-known consequences of this fact (e.g., \cite{AzIobook}, Theorem 3.2.1) are discussed concerning the existence of invariant maximal semidefinite subspaces, which are in fact uniformly definite.

In Section 3 the main results are presented in more detail. The proof makes use of the factorization approach of \cite{BGKOT1}. In particular, a large role is played by the main operator $A^\times$ of the inverse system. It is shown that this operator too is a bicontraction in the \krein\ space. Consequently, it is dichotomous as well. A result on matching of invariant subspaces of $A$ and $A^\times$ is then proved, and that is used to prove the main results.

The following notation will be used throughout the paper. We denote by $\cB(\cX,\cY)$ (by $\cB(\cX)$) the sets of bounded linear operators from the Banach space $\cX$ to the Banach space $\cY$ (linear operators on $\cX$). For an operator $A$, we denote by $\sigma(A)$ the spectrum of $A$, and by $\rho(A)$ the resolvent set of $A$. The symbols $\prec$ and $\succ$ will be used to partially order selfadjoint operators on a Hilbert space $\cX$; more precisely, $H\succeq 0$ or $0\preceq H$ means that $H$ is positive semidefinite, while we write $H\succ 0$ or $0\prec H$ when $H\in\cB(\cX)$ is uniformly positive, that is, for some $\eta>0$, we have 
$$
\langle Hx,x \rangle \geq \eta\, \|x\|^2,\quad x\in\cX .
$$
By writing $H\preceq G$ or $H \prec G$, we mean that $G-H \succeq 0$ or $G-H \succ 0$, respectively. Otherwise, our operator theory notation and terminology follows the standard conventions, cf., \cite{BGKOT1,GGKbook90,GGKOT63}.

Let $\cK$ be a \krein\ space with indefinite inner product $[\cdot ,\cdot ]$. Then $\cK$ can be decomposed as $\cK=\cK_-[\dotplus]\cK_+$, where $\cK_+$ is a Hilbert space and $\cK_-$ is an anti-Hilbert space in the \krein\ space inner product $[\cdot,\cdot]$; here $[\dotplus]$ denotes a direct sum that is orthogonal with respect to that indefinite inner product. Such decompositions are called fundamental decompositions, and the norms associated to different fundamental decompositions via
$$
	\|x_+ + x_-\|_\cK^2:=\langle x_+,x_+\rangle_{\cK_+}-\langle x_-,x_-\rangle_{\cK_-},\qquad
	x_\pm\in\cK_\pm,
$$
are all equivalent, and they turn $\cK$ into a Hilbert space.

For an operator $A$ on $\cK$, the \krein\ space adjoint of $A$ is denoted by $A^{[*]}$. For a subspace $\cM$ of $\cK$, the space $\cM^{[\perp]}$, called the orthogonal companion of $\cM$ in $\cK$, is the set of all vectors $y\in\cK$ such that $[x,y]=0$ for all $x\in\cM$. A subspace $\cM$ of $\cK$ is called positive semidefinite (or negative semidefinite) when $[x,x]\geq 0$ (or $[x,x]\leq 0$) for all $x\in\cM$. The subspace is said to be maximal  semidefinite when it is semidefinite and it is not contained in a strictly larger semidefinite subspace. The subspace $\cM$ is uniformly positive when $[x,x]\geq \delta \|x\|^2$ for some $\delta>0$. The subspace $\cM$ is uniformly negative if $-\cM$, the same space but with a change of sign in the inner product, is uniformly positive. A closed uniformly positive (closed uniformly negative) subspace is a Hilbert space (an anti-Hilbert space).

Every invertible selfadjoint operator $H$ on a Hilbert space $\cX$ induces a \krein\ space structure on $\cX$ via the indefinite inner product $[x,y]=\langle Hx,y \rangle,$ $x,y\in\cX$. Then $H$ is called the Gram operator of the induced \krein\ space, and for an operator $S\in\cB(\cX)$, the adjoint in the indefinite inner product becomes $S^{[*]}=H^{-1}S^*H$.

%%%%%%%%%%%%%%%%%%%%%%%%%%%%%%%%%%

\section{The strict bounded real lemma for dichotomous systems and associated \krein\ spaces}\label{sec:BRL-Krein}

In this section, we recall the strict bounded real lemma for dichotomous systems from \cite{BGtH18c}, and we present some new results in which we associate \krein\ spaces with the selfadjoint operators that appear from the bounded real lemma. First we give a short review of infinite dimensional dichotomous systems, based on Section 2 of \cite{BGtH18c}.   

\subsection{Dichotomous systems}

Consider a linear discrete-time input-state-output system $\Sigma$ over the nonnegative integers $\Z_+$ given by 
\begin{equation}\label{LinSys}
\Sigma:\quad\left\{
\begin{aligned}
x(n+1) &= A x(n)+ Bu(n),\quad  x(0)=x_0,\\	
y(n) &= Cx(n)+Du(n),\quad n\in\zplus.
\end{aligned}
\right.
\end{equation}
Here the input sequence $(u(n))_{n\in\Z_+}$, state sequence $(x(n))_{n\in\Z_+}$ and output sequence $(y(n))_{n\in\Z_+}$ take values in Hilbert spaces $\cU$, $\cX$ and $\cY$, called the input, state and output spaces, respectively, $x_0\in\cX$ is the given initial state at time $n=0$, and $A\in\cB(\cX)$, $B\in\cB(\cU,\cX)$, $C\in\cB(\cX,\cY)$, $D\in\cB(\cU,\cY)$ are the main, control, observation and feedthrough operators, respectively. A triple $(u(n),x(n),y(n))_{n\in\Z_+}$ is called a \emph{trajectory} of \eqref{LinSys}, if $u(n)\in\cU$, $x(n)\in\cX$, $y(n)\in\cY$ for each $n\in\Z_+$, and \eqref{LinSys} is satisfied. Sometimes we identify the system $\Sigma$ with the {\em system colligation} given by the block operator matrix $\sbm{A&B\\C&D}$, also called the {\em system matrix} of $\Sigma$. The transfer function of $\Sigma$ is given by 
\begin{equation}\label{Transfer}
F_\Sigma(z)=D+zC(I-zA)^{-1}B \mbox{ for $z\in \rho(A)^{-1}\cup\{0\}$.}  
\end{equation}
It should be noted that although the transfer function is formally defined only for $z\in  \rho(A)^{-1}\cup\{0\}$, it may happen that it has an analytic continuation to a larger domain. We shall use for this analytic continuation, if it exists, the same notation $F_\Sigma(z)$. The adjoint system $\Sigma^*$ of $\Sigma$ is the linear system with system matrix $\sbm{A&B\\C&D}^*=\sbm{A^*&C^*\\B^*&D^*}$. 

The system $\Sigma$ is called {\em dichotomous} if the main operator $A$ is {\em dichotomous}, meaning that its spectrum $\sigma(A)$ is disjoint from the complex unit circle $\T$, or, equivalently, there is a direct sum decomposition $\cX=\cX_- \dotplus \cX_+$ of the state space such that
\begin{equation}\label{eq:dichotdecomp}
	A=\bbm{A_-&0\\0&A_+}: \cX_-\dotplus\cX_+\to\cX_-\dotplus\cX_+,
\end{equation}
with $A_-=A|_{\cX_-}\in\cB(\cX_-)$ with $\sigma(A_-)\subset\E:=\C\setminus\overline\D$ and $A_+=A|_{\cX_+}\in\cB(\cX_+)$ with $\sigma(A_+)\subset\D$. In particular, $A_-$ is invertible, and $A_-^{-1}$ and $A_+$ are stable. Observe that these subspaces are uniquely determined by $A$; in fact $\cX_-$ and $\cX_+$ are the spectral subspaces of $A$ corresponding to $\E$ and $\D$, respectively. We then call $(\Xscr_-,\Xscr_+)$ the \emph{dichotomous pair (of subspaces)} of $A$. More precisely, let 
$$
	P_+:=\frac{1}{2\pi i}\int_\mathbb{T} (zI-A)^{-1} \ud z
$$ 
be the Cauchy integral of the resolvent of $A$ over the unit circle (counterclockwise); then $P_+$ is the so-called \emph{spectral projection} of $A$ corresponding to the unit disc. Its image ${\rm Im\, }P_+$ is the subspace $\cX_+$ and its kernel ${\rm Ker\,}P_+$ is the subspace $\cX_-$.

\subsection{Realization}\label{realization}

If $G(z)$ is the transfer function of a dichotomous system, then it is clear that $G(z)$ is analytic on a neighborhood of the unit circle $\T$, as well as on a neighborhood of zero. 
In this subsection we show that the converse also holds.
 Let $G(z)$ be a $\cB(\cU,\cY)$-valued function, for Hilbert spaces $\cU$ and $\cY$, which is analytic on a neighborhood of the unit circle $\T$ as well as on a neighborhood of $0$. We claim that $G(z)$ has a realization of the form
\begin{equation}\label{Greal}
G(z)=D+zC(I-zA)^{-1} B,
\end{equation}
where $A\in\cB(\cX)$ is dichotomous, $B\in\cB(\cU,\cX), C\in\cB(\cX,\cY)$ and $D\in\cB(\cU,\cY)$. To see this, we shall use results from \cite{BGKR08}, in particular, from Section 4.2. First we observe that $G(z)$ can be written as $G(z)=G_i(z)+G_o(z)$, where $G_i(z)$ is analytic inside and on the unit circle, and $G_o(z)$ is analytic outside and on the unit circle with value zero at infinity. 

Let $0<r_i<1<r_o<\infty$ be such, that the annulus $r_o \overline{\D} \cap r_i \overline{\E}$ around $\T$ is contained in the domain of analyticity of $G(z)$. Then $G_i(z)$ is analytic on a neighborhood of $r_o\overline{\D}$ and $G_o(z)$ is analytic on a neighborhood of $r_i\overline{\E}$. In order to determine a realization for $G_i(z)$ we follow the construction in Section 4.2, specifically in the proof of Theorem 4.3 in \cite{BGKR08}, but now carried out on the Hilbert space $L^2 (r_o \T, \cY)$, the space of square integrable measurable functions on the contour $r_o\T$ with values in $\cY$, rather than the space $C(r_o \T,\cY)$ of continuous functions on $r_o\T$ with values in $\cY$. This leads to a realization of $G_i(z)$ of the form
\begin{equation}\label{eq:GiReal}
	G_i(z)=\delta_i+ \gamma_i(z I -\alpha_i)^{-1}\beta_i,
\end{equation}
with $\delta_i\in\cB(\cU,\cY)$, $\gamma_i\in\cB(\cX_i,\cY)$, $\alpha_i\in\cB(\cX_i)$, $\beta_i\in\cB(\cU,\cX_i)$, where $\cX_i=L^2 (r_o \T, \cY)$, and $\sigma(\alpha_i)=r_0\T$, so that
 $r_o\D \subset \rho(\alpha_i)$. In particular, $\alpha_i$ is invertible and we can rewrite \eqref{eq:GiReal} as 
\begin{align*}
G_i(z)&=\delta_i-\gamma_i(I-z \alpha_i^{-1})^{-1}\alpha_i^{-1}\beta_i\\ 
&= \delta_i - \gamma_i\alpha_i^{-1}\beta_i - z\gamma_i\alpha_i^{-1}(I-z \alpha_i^{-1})^{-1}\alpha_i^{-1}\beta_i.
\end{align*}

For $G_o(z)$ we consider the function $\widehat{G}_o(z):= G_o(1/z)$ which is analytic on a neighborhood of $r_i^{-1}\overline{\D}$ and apply in the same way the techniques from the proof of \cite[Theorem 4.3]{BGKR08}, modified as above, leading to a realization of $\widehat{G}_o(z)$ of the form 
\[
\widehat{G}_o(z)=\delta_o+ \gamma_o(z I -\alpha_o)^{-1}\beta_o
\]
with $\delta_o\in\cB(\cU,\cY)$, $\gamma_o\in\cB(\cX_o,\cY)$, $\alpha_o\in\cB(\cX_o)$, $\beta_o\in\cB(\cU,\cX_o)$, where $\cX_o=L^2 (r_i^{-1} \T, \cY)$, and $\sigma(\alpha_o)=r_i^{-1}\T$, so that
$r_i^{-1}\D \subset \rho(\alpha_o)$. For $G_o(z)$ We then have
\[
G_o(z)=\delta_o+ \gamma_o(z^{-1} I -\alpha_o)^{-1}\beta_o=\delta_o+ z\gamma_o(I -z\alpha_o)^{-1}\beta_o.
\]
Combining the two realizations for $G_i(z)$ and $G_o(z)$ above, we see that $G(z)$ admits a realization of the form \eqref{Greal} with 
\[
D=\delta_o+\delta_i-\gamma_i\alpha_i^{-1}\beta_i,\ \ C=\mat{-\gamma_i\alpha_i^{-1} & \gamma_o},\ \ A=\mat{\alpha_i^{-1}&0\\0&\alpha_o},\ \, B=\mat{\alpha_i^{-1}\beta_i\\ \beta_o}.
\]

Since $\sigma(\alpha_i)=r_o\T$, it follows that $\sigma(\alpha_i^{-1})=r_o^{-1}\T$, and in particular $\alpha_i^{-1}$ is stable. Then $\sigma(A)=r_o^{-1}\T\cup r_i^{-1}\T$, which does not intersect $\T$; thus $A$ is dichotomous.

\subsection{The strict Bounded Real Lemma}

We shall be interested in the situation where the supremum norm of the transfer function $F_\Sigma$ on $\T$ is strictly less than one, that is, $\|F_\Sigma\|_{\infty,\T}:=\sup_{z\in\T} \|F_\Sigma(z)\|<1$. Note that the supremum is actually a maximum by continuity of $F_\Sigma$ and compactness of $\T$. The strict Bounded Real Lemma provides a necessary and sufficient criterium for this to happen, in terms of an operator matrix inequality of Kalman-Yakubovich–Popov (KYP) type; the next result is Theorem 7.1 from \cite{BGtH18c}.

\pagebreak

\begin{theorem}\label{T:dichotSBRL}
For a dichotomous system \eqref{LinSys}, the transfer function \eqref{Transfer} satisfies $\|F\|_{\infty,\T}<1$ if and only if there exists an invertible, selfadjoint operator $H$ in $\cB(\cX)$ which solves the \emph{strict KYP inequality} for $\Sigma$, i.e., 
\begin{equation}\label{eq:KYPdichotomous}
	\bbm{A&B\\C&D}^*\bbm{H&0\\0&I_\cY}\bbm{A&B\\C&D} \prec \bbm{H&0\\0&I_\cU}.
\end{equation}
In this case, the dimensions of the spectral subspaces of $A$ over $\D$ and $\E$ correspond to the dimensions of the spectral subspaces of $H$ over $(0,\infty)$ and $(-\infty,0)$, respectively, in the sense that they are isomorphic as subspaces of $\cX$.
\end{theorem}

\subsection{The associated \krein\ spaces and bicontractivity of $\Sigma$}

Set $\Sigma=\sbm{A&B\\C&D}$, and assume that the strict KYP inequality \eqref{eq:KYPdichotomous} is satisfied for an invertible and selfadjoint operator $H$. We define the \krein\ space $\cK^\cU$, which is equal to $\cX\oplus\cU$ with the indefinite inner product defined by the Gram operator $H\oplus I_{\cU}$. 
Likewise, introduce the \krein\ space $\cK^\cY$ which is
equal to $\cX\oplus I_{\cY}$ with the Gram operator given by $H\oplus I_{\cY}$. The strict KYP inequality then says that $\Sigma$ is a uniform contraction from the \krein\ spaces $\cK^\cU$ to the \krein\ space $\cK^\cY$,
 that is, there exists an $\varepsilon>0$ such that 
\begin{equation}\label{eq:strictcontr}
[\Sigma f, \Sigma f]_{\cK^\cY} \leq [f, f]_{\cK^\cU} - \varepsilon \|f\|_{\cX\oplus\cU}^2, \quad f\in \cX\oplus \cU. 
\end{equation}

We would like to establish that $H^{-1}$ solves the strict KYP inequality for $\Sigma^*$, as is the case when the spectrum of $A$ is in $\D$, cf., \cite{BGtH18a}. To show that $H^{-1}$ solves the strict KYP inequality for $\Sigma^*$ is equivalent to show that the \krein\ space adjoint $\Sigma^{[*]}=(H^{-1}\oplus I_{\cU})\Sigma^*(H\oplus I_{\cY})$ is a uniform contraction from $\cK^\cY$ to $\cK^\cU$, i.e., to show that there exists a $\delta>0$ such that
\[
[\Sigma^{[*]} g, \Sigma^{[*]} g]_{\cK^\cU} \leq [g, g]_{\cK^\cY} - \delta\, \|g\|_{\cX\oplus\cY}^2, \quad g\in \cX\oplus \cY. 
\] 
In other words, we need to show that $\Sigma$ is a \emph{uniform \krein\ space bicontraction} from $\cK^\cU$ to $\cK^\cY$, and for this, we will use \cite[Prop.\ 2.4.26]{AzIobook}.

\begin{theorem}\label{T:BRL*}
Let $H$ be an invertible and selfadjoint solution to the strict KYP inequality \eqref{eq:KYPdichotomous} for a dichotomous system $\Sigma$. Then $H^{-1}$ solves the strict KYP inequality for $\Sigma^*$, that is
\begin{equation}\label{eq:KYPdichotadj}
	\bbm{A&B\\C&D}\bbm{H^{-1}&0\\0&I_\cU}\bbm{A&B\\C&D}^* \prec\bbm{H^{-1}&0\\0&I_\cY}.
\end{equation}
\end{theorem}

\begin{proof}[\bf Proof] 
Let $\cX=\cX_- \dotplus \cX_+$ be the dichotomous decomposition of the state space of $\Sigma$, so that the operators $A$, $B$ and $C$ decompose as
\begin{equation}\label{eq:dichotpf}
A=\mat{A_-&0\\0&A_+},\quad B=\mat{B_-\\B_+},\quad C=\mat{C_-&C_+},   
\end{equation}
with $\sigma(A_+)$ in $\D$ and $\sigma(A_-)$ in $\E$. 
As before, let $P_+$ denote the bounded projection in $\cX$ onto $\cX_+$ along $\cX_-$, and let $\cX_\pm$ inherit the inner product of $\cX$. Then the linear operator
$$
	P:=\bbm{I-P_+\\P_+}:\cX\to\cX_-\oplus\cX_+
$$
is bounded, injective and surjective, and hence it has an inverse. Here $\cX_-\oplus\cX_+$ is now an orthogonal direct sum.

Define 
\begin{equation}\label{eq:SigmaHatDef}
	\widehat\Sigma
	=\bbm{\widehat{A} & \widehat{B} \\ \widehat{C} & D}
	:=\bbm{P&0\\0&I_\cY}\Sigma\bbm{P&0\\0&I_\cU}^{-1}:
	\bbm{\cX_-\oplus\cX_+\\\cU}\to
	\bbm{\cX_-\oplus\cX_+\\\cY},
\end{equation}
and note that $\widehat{A}$ decomposes as
$$
	\widehat A=PAP^{-1}=\mat{A_-&0\\0& A_+}:
	\cX_-\oplus\cX_+\to\cX_-\oplus\cX_+.
$$
In fact, $\widehat A=A$ if one identifies $\cX$ with $\cX_-\oplus\cX_+$ via the isomorphism $P$, but we refrain from making this identification below. Due to the orthogonality of the decomposition $\cX_-\oplus\cX_+$, the adjoint $\widehat A^*=\sbm{A_-^*&0\\0&A_+^*}$ is also diagonal, whereas $A^*$ as an operator on $\cX=\cX_-\dotplus\cX_+$ is in general not.

Next let $H$ solve the strict KYP inequality for $\Sigma$ and observe that then $\widehat H:=P^{-*}HP^{-1}$ solves the strict KYP inequality for $\widehat\Sigma$.
 Now decompose $\widehat H$ as
\[
\widehat H=\mat{H_-&H_0\\H_0^* & H_+}:
\cX_-\oplus\cX_+\to\cX_-\oplus\cX_+,
\]
and let $\varepsilon >0$ be such that $\widehat H-\widehat A^*\widehat H\widehat A\succeq\varepsilon \,I_{\cX_-\oplus\cX_+}$. Then, following the arguments on pages 56 and 57 of \cite{BGtH18c}, we obtain that $H_-\prec 0$ in $\cX_-$ and $H_+ \succ 0$ in $\cX_+$. In particular $H_-$ and $H_+$ are invertible. Let $\widehat H/H_+:=H_- - H_0 H_+^{-1}H_0^*$ be the Schur complement of $\widehat H$ with respect to $H_+$. Then, since $H_- \prec 0$ and  $H_+^{-1}\succ 0$, we have $\widehat H/H_+\prec 0$.
Moreover,
\[
T:=\mat{I_{\cX_-}&0\\H_+^{-1}H_0^*&I_{\cX_+}}:
\cX_-\oplus\cX_+\to\cX_-\oplus\cX_+ 
\] 
is invertible and we have
\begin{equation}\label{Hdiag}
\widehat H=T^* \mat{\widehat H/H_+&0\\0& H_+} T.
\end{equation} 
Now set 
\begin{equation}\label{eq:SigmaTildeDef}
\widetilde\Sigma =\bbm{\widetilde A & \widetilde B\\ \widetilde C & D}
:=\bbm{T&0\\0&I_\cY}\widehat\Sigma\bbm{T&0\\0&I_\cU}^{-1}:
	\bbm{\cX_-\oplus\cX_+\\\cU}\to
	\bbm{\cX_-\oplus\cX_+\\\cY}
\end{equation}
in order to obtain
\begin{equation}\label{eq:AtildeBlock}
\wtil{A}=T\widehat A T^{-1}=\mat{A_-&0\\ H_+^{-1}H_0^* A_- - A_+ H_+^{-1}H_0^*& A_+}:
\cX_-\oplus\cX_+\to\cX_-\oplus\cX_+.
\end{equation}
Moreover,
$$
	\widetilde H:=
	T^{-*}\widehat HT^{-1}
	=\mat{\widehat H/H_+&0\\0& H_+}:
	\cX_-\oplus\cX_+\to\cX_-\oplus\cX_+
$$
solves the strict KYP inequality for $\widetilde\Sigma$, with the state space $\cX_-\oplus\cX_+$.

We next prove that $\widetilde H^{-1}$ solves the strict KYP inequality for $\widetilde\Sigma^*$. For this, observe that $\widetilde K=\sbm{\Xscr_-\\\zero}[\dotplus]\sbm{\zero\\\Xscr_+}$ is a fundamental decomposition of the \krein\ space $\widetilde K$ formed by equipping the Hilbert space $\cX_-\oplus\cX_+$ with the indefinite inner product with Gram operator $\widetilde H$.  Indeed, we have 
$$
	\left[\bbm{x_-\\0},\bbm{0\\x_+}\right]_{\widetilde K}=
	\Ipdp{\bbm{\widehat H/H_+&0\\0&H_+}\bbm{x_-\\0}}{\bbm{0\\x_+}}_{\cX_-\oplus\cX_+}=0,
$$
and it follows from $\widehat H/H_+\prec 0$ that $\sbm{\Xscr_-\\\zero}$ is the anti-Hilbert space, while $H_+\succ 0$ implies that $\sbm{\zero\\\Xscr_+}$ is the Hilbert space in the fundamental decomposition. 

Let $\widetilde K^\cU$ be the \krein\ space obtained by equipping $(\cX_-\oplus\cX_+)\oplus\cU$ with the Gram operator $\sbm{\widetilde H&0\\0&I_\cU}$ and define $\widetilde K^\cY$ analogously. Then we get the fundamental decompositions
$$
	\widetilde K^\cU=\widetilde K^\cU_-[\dotplus]\widetilde K^\cU_+:=
	\bbm{\Xscr_-\\\zero}[\dotplus]\bbm{\Xscr_+\\\cU}
	\ \text{and}\
	\widetilde K^\cY=\widetilde K^\cY_-[\dotplus]\widetilde K^\cY_+:=
	\bbm{\Xscr_-\\\zero}[\dotplus]\bbm{\Xscr_+\\\cY},
$$
with $\sbm{\Xscr_-\\\zero}$ being the anti-Hilbert spaces (the zero vectors are from $\cU$ and $\cY$, respectively). By the discussion before the theorem, $\widetilde \Sigma$ is a contraction from $\widetilde K^\cU$ into $\widetilde K^\cY$, and according to \cite[Cor.\ 8.1.7]{GheoBook22}, $\widetilde\Sigma$ is in fact even a bicontraction because
$$
	P_-\widetilde \Sigma\big|_{\widetilde K^\cU_-}=A_-
$$
by \eqref{eq:AtildeBlock}, and this operator is invertible on $\cX_-$; here $P_-$ denotes the projection in $\widetilde K^\cY$ onto ${\widetilde K^\cY_-}$ along ${\widetilde K^\cY_+}$. 

By the previous paragraph, $\widetilde\Sigma$ is a bicontraction from $\widetilde K^\cU$ to $\widetilde K^\cY$, so that
$$
	\mat{\wtil{A}&\wtil{B}\\\wtil{C}&\wtil{D}}
	\mat{\wtil{H}^{-1}&0\\0&I_\cU}
	\mat{\wtil{A}&\wtil{B}\\\wtil{C}&\wtil{D}}^*
	\preceq
	\mat{\wtil{H}^{-1}&0\\0&I_\cY},
$$
and combining this with $\widetilde H^{-1}=TPH^{-1}(TP)^*$, \eqref{eq:SigmaHatDef} and \eqref{eq:SigmaTildeDef}, we get that $H^{-1}$ satisfies the (non-strict) KYP inequality for $\Sigma^*$. Finally, we apply a Schur coupling argument in order to prove that $H^{-1}$ in fact solves the \emph{strict} KYP inequality for $\Sigma^*$.

We have so far established that
\[
	\bbm{H^{-1}&0\\0& I_\cY} - \Sigma \bbm{H^{-1}&0\\0& I_\cU} \Sigma^* \succeq 0
\]
and we need to show that the left hand side is also invertible. For this purpose, consider the following $2 \times 2$ block operator 
\[
L:=\mat{\bbm{H^{-1}&0\\0& I_\cY} & \Sigma\\ 
	\Sigma^* & \bbm{H&0\\0& I_\cU}}. 
\]
Note that the left upper and right lower blocks are invertible, and that the Schur complements with respect to these blocks are given by 
\[
\bbm{H&0\\0& I_\cU} - \Sigma^* \bbm{H&0\\0& I_\cY} \Sigma
\quad\mbox{and}\quad 
\bbm{H^{-1}&0\\0& I_\cY} - \Sigma \bbm{H^{-1}&0\\0& I_\cU} \Sigma^*.
\]
Hence these two operators are Schur coupled. Since $H$ solves the strict KYP inequality for $\Sigma$, the first of these Schur complements is uniformly positive, and, in particular, invertible, while we have shown that the second Schur complement is positive semidefinite. However, since the operators are Schur coupled, the second Schur complement must also be invertible; see, e.g., \cite{BGKR05,BT94}, and hence it is even uniformly positive. This precisely means that $H^{-1}$ solves the strict KYP inequality for $\Sigma^*$.
\end{proof}

Note that \eqref{eq:KYPdichotomous} and \eqref{eq:KYPdichotadj} imply that $A^*HA \prec H$ and $AH^{-1}A^{*} \prec H^{-1}$, respectively. In other words, $A$ is a uniform bicontraction on the \krein\ space $\cK$ with Gram operator $H$. It is also true that uniformly bicontractive operators are dichotomous. The following result can be found by puzzling together various parts of \cite{AzIobook}, and the proof has precise references.

\begin{lemma}\label{lem:bicontractdichotAlt}
Let $A\in\cB(\cX)$ and let $H\in\cB(\cX)$ be selfadjoint and invertible, so that the Hilbert space $\cX$ becomes a \krein-space $\cK$ with Gram operator $H$. Assume that $A$ is uniformly bicontractive in $\cK$. Then $A$ is dichotomous on $\cX$, and $A$ has a unique pair $(\cX_-,\cX_+)$ of invariant subspaces, such that $\pm\cX_\pm$ are maximal positive semidefinite subspaces of $\cK$ and $\cX=\cX_-\dotplus\cX_+$. In fact, $(\cX_-,\cX_+)$ is the dichotomous pair of $A$, and $\pm\cX_\pm$ are Hilbert spaces in the inner product inherited from $\cK$.
\end{lemma}

\begin{proof}[\bf Proof] 
The operator $A$ is uniformly bicontractive in $\cK$ if and only if it is uniformly biexpansive in the \krein\ space $-\cK$ by \cite[Def. 2.4.23]{AzIobook}, and by \cite[Thm 2.4.31]{AzIobook} it then holds that $\sigma(A)\cap\T=\emptyset$. The invariance, existence and uniqueness of $\cX_\pm$ follows from \cite[Thm 3.2.1]{AzIobook}, and the fact that $\Xscr=\cX_-+\cX_+$ is given in (2.2) in the proof of that result. Theorem 3.2.1 in \cite{AzIobook} also states that $\cX_-$ is uniformly positive and $\cX_+$ uniformly negative in $-\cK$, and since every maximal semidefinite subspaces is closed, this implies that $\cX_+$ and $-\cX_-$ are Hilbert spaces in the inner product inherited from $\cK$. From that fact, it follows that $\cX_-\cap\cX_+=\zero$, since every vector in this intersection satisfies $\|x\|^2_{\cX_+}=0$, so that $\Xscr=\cX_- +\cX_+=\cX_-\dotplus\cX_+$. 

By the proof of \cite[Thm 3.2.1]{AzIobook}, it moreover holds that $\sigma(A_1)\subset \overline\E$ and $\sigma(A_2)\subset\overline\D$, where
\begin{equation*}\label{eq:KreinDicho}
	A=\bbm{A_1&0\\0&A_2}:\cX_-\dotplus\cX_+\to \cX_-\dotplus\cX_+.
\end{equation*}
Since $\sigma(A)=\sigma(A_1)\cup\sigma(A_2)$ does not intersect $\T$, we in fact have $\sigma(A_1)\subset \E$ and $\sigma(A_2)\subset\D$. Then $A$ is dichotomous, and by the uniqueness of dichotomous pairs, $(\cX_-,\cX_+)$ is the dichotomous pair of $A$.
\end{proof}

\section{Wiener-Hopf factorization of functions of the form identity plus a strictly contractive transfer function}

In this section we prove our main results, which are the following two theorems. 

\begin{theorem}[right canonical Wiener-Hopf factorization]\label{thm:rightWH}
Consider a $\cB(\cU)$-valued function $G(z)$, with $\cU$ a Hilbert space, of the form
\begin{equation*}%\label{eq:Fspecial}
	 G(z)=I+F(z),\qquad
	F(z)=D+zC(I-z A)^{-1} B,\qquad
	\|F\|_{\infty,\T}<1.
\end{equation*}
Assume that $A$ is dichotomous and that $I+D$ is invertible. Then $A^\times :=A-B(I+D)^{-1}C$ is also dichotomous.
Let $(\cX_-,\cX_+)$ and $(\cX^\times_-,\cX^\times_+)$ be the dichotomous pairs of $A$ and $A^\times$, respectively. Then
\[
\cX=\cX_- \dot+ \cX^\times_+.
\]
Write $\Pi_r$ for the projection in $\cX$ onto $\cX^\times_+$ along $\cX_-$. Take any factorization $I+D=D_1D_2$ with $D_1$ and $D_2$ invertible. 
Then $G(z)=V_-(z)\,V_+(z)$, where
$$
\begin{aligned}
	V_{-}(z) &:= D_1+zC\left(I-zA\right)^{-1}(I-\Pi_r)BD_2^{-1}
	\qquad\text{and}\\
	V_{+}(z) &:= D_2+zD_1^{-1}C\Pi_r
	\left(I-zA\right)^{-1}B,
	\quad z\in\rho(A)^{-1}\cup \{0\},
\end{aligned}
$$
with inverses given by
$$
\begin{aligned}
	V_{-}(z)^{-1} &= D_1^{-1}-zD_1^{-1}C(I-\Pi_r)\left(I-zA^\times\right)^{-1}BD_2^{-1}D_1^{-1}
	\qquad\text{and}\\
	V_{+}(z)^{-1} &= D_2^{-1}-zD_2^{-1}D_1^{-1}C
	\left(I-zA^\times\right)^{-1}\Pi_r BD_2^{-1},
	\quad z\in\rho(A^\times)^{-1}\cup \{0\}.
\end{aligned}
$$
The functions $V_+(z)$ and $V_+(z)^{-1}$ extend analytically to functions on a neighborhood of the closed unit disc $\overline \D$, while $V_-(z)$ and $V_-(z)^{-1}$ extend analytically to functions on a neighborhood of the closed complement $\overline \E$ of $\D$.
\end{theorem}

By symmetry, we have the following analogue for left canonical factorization, obtained simply by choosing a different projection $\Pi_\ell$.

\begin{theorem}[left canonical Wiener-Hopf factorization]\label{thm:leftWH}
Let $G(z)$, $I+D=D_1D_2$ and the dichotomous pairs be as in Theorem \ref{thm:rightWH}. Then also $\cX=\cX^\times_-\dot+ \cX_+$. Write $\Pi_\ell$ for the projection in $\cX$ onto $\cX^\times_-$ along $\cX_+$. Then $G(z)=W_+(z)\,W_-(z)$, where
$$
\begin{aligned}
	W_{+}(z) &:= D_1+zC\left(I-zA\right)^{-1}(I-\Pi_\ell)BD_2^{-1}
	\qquad\text{and}\\
	W_{-}(z) &:= D_2+zD_1^{-1}C\Pi_\ell
	\left(I-zA\right)^{-1}B,
	\quad z\in\rho(A)^{-1}\cup \{0\},
\end{aligned}
$$
with inverses given by
$$
\begin{aligned}
	W_{+}(z)^{-1} &= D_1^{-1}-zD_1^{-1}C(I-\Pi_\ell)\left(I-zA^\times\right)^{-1}BD_2^{-1}D_1^{-1}
	\qquad\text{and}\\
	W_{-}(z)^{-1} &= D_2^{-1}-zD_2^{-1}D_1^{-1}C
	\left(I-zA^\times\right)^{-1}\Pi_\ell BD_2^{-1},
	\quad z\in\rho(A^\times)^{-1}\cup \{0\}.
\end{aligned}
$$
The functions $W_+(z)$ and $W_+(z)^{-1}$ extend analytically to functions on a neighborhood of the closed unit disc $\overline \D$, while $W_-(z)$ and $W_-(z)^{-1}$ extend analytically to functions on a neighborhood of the closed complement $\overline \E$ of $\D$.
\end{theorem}

The analytic extensions will be addressed and made explicit in the proof of the main theorem below. In particular, for the functions $V_+(z)$, $V_-(z)$, $W_+(z)$ and $W_-(z)$ and their inverses formulas will be given which display the analytic extensions clearly. 

For the proof of the main theorems, we need several lemmas.

\begin{lemma}\label{lem:UnifBiExp}
Assume that the transfer function $F(z)$ of the dichotomous system \eqref{LinSys} satisfies $\|F\|_{\infty,\T}<1$ and let $H\in\cB(\cX)$ be an invertible selfadjoint solution to the KYP inequality \eqref{eq:KYPdichotomous}.  If additionally $\cY=\cU$ and $I+D$ is invertible, then the operator
\begin{equation}\label{eq:AtimesDef}
	A^\times=A-B(I+D)^{-1}C
\end{equation}
is uniformly bicontractive in the \krein\ space $\cK$ obtained by pairing $\cX$ with the Gram operator $H$. 
\end{lemma}

\begin{proof}[\bf Proof]
Assume that $I+D$ is invertible. Multiplying \eqref{eq:KYPdichotomous} from the left by $\sbm{I\\-(I+D)^{-1}C}^*$ and from the right by $\sbm{I\\-(I+D)^{-1}C}$, 
we get
{\small
\[
	\bbm{ A^\times\\ (I+D)^{-1}C}^*	
	\bbm{H&0\\0&I}
	\bbm{ A^\times\\ (I+D)^{-1}C}
	\prec
	\bbm{I_\cX\\-(I+D)^{-1}C}^*\bbm{H&0\\0&I}\bbm{I_\cX\\-(I+D)^{-1}C}.
\]	
}
This yields
\[
	( A^\times)^* H A^\times+C^*(I+D)^{-*}(I+D)^{-1}C
	\prec
	H+C^*(I+D)^{-*}(I+D)^{-1}C .
\]
In turn, this is equivalent to $( A^\times)^* H A^\times \prec H$, i.e., $A^\times$ is uniformly contractive in $\cK$. 
By Theorem \ref{T:BRL*}, $H^{-1}$ is a solution to the KYP inequality \eqref{eq:KYPdichotadj}. Multiplying \eqref{eq:KYPdichotadj} from the left by the operator $\bbm{I&-B(I+D)^{-1}}$ and from the right by $\bbm{I&-B(I+D)^{-1}}^*$, we get in the same way that $ (A^\times)^{[*]}$ is also uniformly contractive in $\cK$.
\end{proof}

Next we provide sufficient conditions for $A^\times$ to be dichotomous, and for the dichotomous pairs $(\cX_-,\cX_+)$ and $(\cX_-^\times,\cX_+^\times)$ of $A$ and $A^\times$, respectively, to \emph{match} in the sense that
\begin{equation}\label{eq:match}
	\cX=\cX_\pm\dotplus\cX_\mp^\times.
\end{equation}
The proof will use the concept of a Banach limit, which we briefly recall here for self-containment. Let $\ell^\infty$ denote the Banach space of bounded sequences $\zplus\to\C$ together with the supremum norm and let $c\subset \ell^\infty$ be the 
closed subspace consisting of the convergent sequences in $\ell^\infty$. A \emph{Banach limit} is any continuous linear functional $\phi\in(\ell^\infty)^*$ with the following four properties:
\begin{enumerate}
\item $\phi(\bx)=\lim_{n\to\infty} \bx(n)$ if $\bx\in c$,
\item $\|\phi\|=1$,
\item $\phi(S\bx)=\phi(\bx)$, where $S:\ell^\infty\to\ell^\infty$ is the left shift $(S \bx)(n)=\bx(n+1)$, $n\in\zplus$\!\!, and
\item if $\bx(n)\geq0$ for all $n\in\zplus$, then $\phi(\bx)\geq0$.
\end{enumerate}
Property (4) makes comparison possible:\ If $\bx(n)\leq \by(n)$ for all $n\in\zplus$, then $\phi(\bx)\leq\phi(\by)$. Banach limits exist, by the Hahn-Banach theorem, but there does not exist a unique Banach limit. For the sake of the proof of the next lemma, it does not matter which Banach limit we choose, so we just fix one.  

\begin{lemma}\label{lem:matching}
Assume that the transfer function $F(z)$ of the dichotomous system \eqref{LinSys} satisfies $\|F\|_{\infty,\T}<1$. If $\cY=\cU$ and $I+D$ is invertible, then $A^\times$ in \eqref{eq:AtimesDef} is dichotomous. Moreover, the dichotomous pair $(\cX_-,\cX_+)$ of $A$ matches the dichotomous pair $(\cX_-^\times,\cX_+^\times)$ of $A^\times$,i.e., equation \eqref{eq:match} holds.
\end{lemma}

\begin{proof}[\bf Proof]
Let $H\in\cB(\cX)$ be an invertible and selfadjoint operator, such that \eqref{eq:KYPdichotomous} holds. Write $\cK$ for the \krein\ space obtained by equipping $\cX$ with the indefinite inner product with Gram operator $H$. By Lemma \ref{lem:UnifBiExp}, the operator $A^\times$ is uniformly bicontractive in $\cK$. By Lemma \ref{lem:bicontractdichotAlt}, the operator $A^\times$ is dichotomous, and in the dichotomous pairs $(\cX_-,\cX_+)$ and $(\cX_-^\times,\cX_+^\times)$ of $A$ and $A^\times$, respectively, the spaces $\pm\cX_\pm$ and $\pm\cX_\pm^\times$ are all maximal positive semidefinite subspaces of $\cK$, with
\begin{equation}\label{eq:Xdichotdecomp}
	\cX=\cX_-\dotplus\cX_+=\cX_-^\times\dotplus\cX_+^\times;
\end{equation}
moreover, $\pm\cX_\pm$ and $\pm\cX_\pm^\times$ are Hilbert spaces with respect to the \krein\ space inner product of $\cK$. 
It then follows from \cite[Cor.\ 1.8.14]{AzIobook} that  $-(\cX_+)^{[\perp]}$ and $(\cX_-^\times)^{[\perp]}$ are Hilbert subspaces of $\cK$, with $[\perp]$ indicating the \krein\ space orthogonal complement. This again implies that $(\cX_+)^{[\perp]}\cap(\cX_-^\times)^{[\perp]}=\{0\}$. In order to establish that $\cX_++\cX^\times_-$ is dense in $\cX$, let now $y\perp (\cX_++\cX^\times_-)$. Set $x=H^{-1}y$, so that $Hx\perp (\cX_++\cX^\times_-)$, i.e., $x\in(\cX_+)^{[\perp]}\cap(\cX_-^\times)^{[\perp]}$; then $y=0$, and it follows that $\cX_++\cX^\times_-$ is dense. It remains only to prove that the sum $\cX=\cX_++\cX^\times_-$ is also direct and closed, because then $\cX=\cX_-\dotplus\cX^\times_+$ follows by symmetry. 

In order to obtain that the sum $\cX_++\cX^\times_-$ is closed and direct, we will apply \cite[Lemma 5.2]{Ran82}. For this purpose let $\ell^\infty(\zplus;\cX)$ denote the Banach space of $\cX$-valued bounded sequences over $\zplus$ together with the supremum norm, and, following \cite{Ber62}, set
$$
	\cN:=\set{\bx\in \ell^\infty(\zplus;\cX)\mid \langle\bx,\bx\rangle=0},
$$
where
\begin{equation}\label{eq:BlimitIP}
	\langle\bx,\by\rangle:=\phi\big((\langle \bx(n),\by(n)\rangle_\cX)_{n\in\zplus}\big),\qquad
	\bx,\by\in\ell^\infty(\zplus;\cX),
\end{equation}
is defined using our arbitrarily fixed Banach limit $\phi$. Then one can easily show that \eqref{eq:BlimitIP} defines an inner product on the quotient space $\ell^\infty(\zplus;\cX)/\cN$, whose Hilbert-space completion we denote by $\widetilde\cX$. The Hilbert space $\cX$ is embedded into $\widetilde\cX$ by identifying $x\in\cX$ with the constant sequence with entries all equal to $x$, and it is clear that the restriction of the inner product on $\widetilde\cX$ to (the embedding of) $\cX$ corresponds to the inner product of $\cX$ via
$$
	\langle\iota x,\iota y\rangle=\langle x,y\rangle_\cX,\qquad\text{where}\qquad \iota x=(x,x,x,\ldots)\in\ell^\infty(\zplus;\cX),\quad x\in\cX.
$$ 
Likewise, any subspace $\cM\subset\cX$ can be embedded into $\wtil{\cX}$, and we write $\wtil{\cM}$ for the closure of $\iota\cM$ in $\wtil{\cX}$. 

As in \cite[Section 4]{Ber62}, for an operator $K$ on $\cX$, we note that the operator
$$
	\bx+\cN\mapsto (K \bx(n))_{n\in\zplus}+\cN,\qquad \bx\in \ell^\infty(\zplus;\cX),
$$
maps $\ell^\infty(\zplus;\cX)/\cN$ boundedly into itself, with the same norm as $K$, and we denote by $\widetilde K$ its continuous extension to a linear operator on $\widetilde\cX$, which has norm $\|\wtil{K}\|_{\cB(\wtil{\cX})}=\|K\|_{\cB(\cX)}$. In particular, the Gram operator $H$ extends to an operator $\wtil{H}$ on $\wtil{\cX}$, and it is easily checked that $\wtil{H}$ is also selfadjoint and invertible, and hence the Gram operator of a \krein\ space consisting of the vectors in $\wtil{\cX}$, which we denote by $\wtil{\cK}$.   

By Lemma 5.1 in \cite{Ran82}, it follows that the closures $\wtil{\cX}_+$ and $\wtil{\cX}_-^\times$ in $\wtil{\cX}$ of $\iota\cX_+$ and $\iota\cX_-^\times$, respectively, are such that $\wtil{\cX}_+$ is a maximal positive semidefinite subspace of $\wtil{\cK}$ and $\wtil{\cX}_-^\times$ is a maximal negative semidefinite subspace of $\wtil{\cK}$. 
Recall that $\cX_+$ is unformly positive and $\cX_-$ is uniformly negative in the \krein\ space $\cK$. Taking any fundamental decomposition $\cK=\cK_+ [\dot+ ] \cK_-$, it follows that with respect to this fundamental decomposition 
\[
\cX_+={\rm Im\,} \begin{bmatrix} I \\ R_+\end{bmatrix}, {\rm \ and \ }
\cX_-={\rm Im\,} \begin{bmatrix}  R_-\\ I \end{bmatrix}.
\] 
Here $R_+$ and $R_-$, the so-called angular operators,
are strict contractions because of the uniform positivity of $\cX_+$ and uniform negativity of $\cX_-$.
Notice that 
\[
\wtil{\cX}_+ ={\rm Im\,} \begin{bmatrix} \wtil{I} \\ \wtil{R_+}\end{bmatrix} {\rm \  and\ }
\wtil{\cX}_- ={\rm Im\,} \begin{bmatrix}\wtil{R_-}\\ \wtil{I}\end{bmatrix},
\]
and that $\wtil{R_\pm}$ are strict contractions. 
It is then easily checked that the restrictions of the indefinite inner product of $\wtil{\cK}$ to $\wtil{\cX}_+$ and $-\wtil{\cX}_-$ makes these subspaces into Hilbert spaces; then $\widetilde\cX_+\cap\widetilde\cX_-^\times=\zero$, and it follows from \cite[Lemma 5.2]{Ran82} that the sum $\cX_++\cX_-^\times$ is both direct and closed. 
\end{proof}

We remark that, in the last paragraph of the preceding proof, $\wtil{\cX}_+$ is moreover invariant under $\wtil{A}$ and $\wtil{\cX}_-^\times$ is invariant under $\wtil{A}^\times$, again by Lemma 5.1 in \cite{Ran82}. 

The following lemma is a  rephrasing of parts of \cite[Theorems 2.1 and 2.8]{BGKR08} specialized to the function $G(z)$ as in our main theorems. Since the proof is essentially the same as the proof in \cite{BGKR08}, we refrain from giving the details here.

\begin{lemma}\label{lem:transferformulas}
Let 
\[
G(z)=I+D+zC(I-zA)^{-1}B,
\]
with $I+D$ invertible. Define $A^\times =A-B(I+D)^{-1}C$.
The following statements hold:
\begin{enumerate}
\item[(i)] Assume that $z\in\rho(A)^{-1}\cup\{0\}$. Then $G(z)$ has an inverse if and only if $z\in\rho(A^\times)^{-1}\cup \{0\}$, 
and in that case, for $z\in\rho(A^\times)^{-1}\cup\{0\}$ we have
\begin{equation}\label{eq:Finv}
	G(z)^{-1} =(I+ D)^{-1}-z (I+D)^{-1}C\left(I-zA^\times\right)^{-1}B(I+D)^{-1}.
\end{equation}

\item[(ii)] Assume that $\cX$ can be decomposed into a (not necessarily orthogonal) direct sum $\cX=\cL\dotplus \cL^\times$, where $A\cL\subset \cL$ and $A^\times \cL^\times\subset \cL^\times$, and let $\Pi$ denote the projection in $\cX$ onto $\cL^\times$ along $\cL$. If  $I+D=D_1D_2$ with $D_1$ and $D_2$ invertible, then $G$ factors as
$$
	G(z)=W_1(z)\,W_2(z),
$$
where
\begin{equation}\label{eq:factorsdef}
\begin{aligned}
	W_1(z) &:= D_1+zC\left(I-zA\right)^{-1}(I-\Pi)BD_2^{-1}
	\qquad\text{and}\\
	W_2(z) &:= D_2+zD_1^{-1}C\Pi
	\left(I-zA\right)^{-1}B,
	\quad z\in\rho(A)^{-1}\cup\{0\}.
\end{aligned}
\end{equation}

\item[(iii)] Under the assumption and with the notation of item (ii), the values of the  functions $W_1(z)$ and $W_2(z)$ in \eqref{eq:factorsdef} are invertible for $z\in \big(\rho(A)^{-1} \cap \rho(A^\times)^{-1}\big)\cup\{0\}$ with inverses given for $z$ in the extended domain $\rho(A^\times)^{-1}\cup\{0\}$ by
\begin{equation}\label{eq:factorsinv}
\begin{aligned}
	W_1(z)^{-1} &= D_1^{-1}-zD_1^{-1}C(I-\Pi)\left(I-zA^\times \right)^{-1}B(I+D)^{-1}
	\qquad\text{and}\\
	W_2(z)^{-1} &= D_2^{-1}-z(I+D)^{-1}C
	\left(I-zA^\times \right)^{-1}\Pi B D_2^{-1}.
\end{aligned}
\end{equation}    

\end{enumerate}
\end{lemma}

We now turn to the proof of the main theorems. In this proof 
the analytic extensions will be addressed and made explicit. In particular, for the functions $V_+(z)$ and $V_-(z)$ and their inverses we have formulas
\eqref{V+form2}, \eqref{V+invform2} and \eqref{V-V-invforms2}, below, where the operators involved are obtained from the block operator formulas of $A$, $A^\times$, $B$ and $C$ with respect to $\cX=\cX_-\dot+\cX^\times_+$ given in \eqref{-+timesDec}, with the analytic extension claims following from \eqref{SpecDec}. Similarly, the claims regarding $W_+(z)$ and $W_-(z)$ and their inverses follow from \eqref{Wformulas}, \eqref{-+timesDec'} and \eqref{SpecDec'}, below.

\begin{proof}[\bf Proof of Theorems \ref{thm:rightWH} and \ref{thm:leftWH}.] 
By assumption, $A$ is dichotomous, and $A^\times$ is dichotomous by Lemma \ref{lem:matching}. The dichotomous pair  $(\cX_-,\cX_+)$ of $A$ provides two invariant (even reducing) subspaces of $A$, and likewise the dichotomous pair $(\cX_-^\times,\cX_+^\times)$ of $A^\times$ gives two invariant subspaces of $A^\times$. From Lemma \ref{lem:matching} we further get the decompositions $\cX=\cX_+\dotplus\cX^\times_-$ and $\cX=\cX_-\dotplus\cX^\times_+$ of $\cX$. In particular, the projections $\Pi_r$ and $\Pi_\ell$ are well-defined.

Applying item (ii) of Lemma \ref{lem:transferformulas} to the decomposition $\cX=\cX_+\dotplus\cX^\times_-$, provides the factorization $G(z)=W_+(z)W_-(z)$ of Theorem \ref{thm:leftWH}, while applying the lemma to $\cX=\cX_-\dotplus\cX^\times_+$ yields the factorization $G(z)=V_-(z)V_+(z)$ of Theorem \ref{thm:rightWH}. In both cases, the formulas for the inverses of the functions in the factorization follow directly from item (iii) of Lemma \ref{lem:transferformulas}.

It remains to prove that the functions $V_+(z)$, $V_+(z)^{-1}$, $W_+(z)$ and $W_+(z)^{-1}$ extend analytically to a neighborhood of $\D$, while $V_-(z)$, $V_-(z)^{-1}$, $W_-(z)$ and $W_-(z)^{-1}$ extend analytically to a neighborhood of $\E$. To see that this is the case for the factorization in Theorem \ref{thm:rightWH} we consider the operators $A$, $A^\times$, $B$ and $C$ as block operator with respect to the decomposition $\cX=\cX_-\dot+\cX^\times_+$ and note that they take the form 
\begin{equation}\label{-+timesDec}
A = \bbm{A_{11}&A_{12} \\ 0&A_{22}},\ \
A^\times = \bbm{A_{11}^\times&0\\A_{21}^\times&A_{22}^\times},\ \
C=\bbm{C_1&C_2},\ \
B=\bbm{B_1\\B_2}.  
\end{equation}
By the properties of the dichotomous pairs $A$ and $A^\times$ it follows that $\sigma(A_{11})\subset\E$ and $\sigma(A_{22}^\times)\subset\D$. Since we also have the decomposition $\cX=\cX_- \dot+ \cX_+$ and $\cX_+$ is also an invariant subspace of $A$, it follows from Lemma 5.9 in \cite{BGKR08} that $A_{22}$ and $A|_{\cX_+}$ are similar, and hence $\sigma(A_{22})=\sigma(A|_{\cX_+})\subset \D$. By an analogous argument it follows that $\sigma(A_{11}^\times)\subset \E$. Hence, we have
\begin{equation}\label{SpecDec}
\sigma(A_{11})\subset \E,\quad \sigma(A_{22})\subset \D,\quad \sigma(A_{11}^\times)\subset \E,\quad \sigma(A_{22}^\times)\subset \D. 
\end{equation}
Since $\sigma(A_{11})\cap \sigma(A_{22})=\emptyset$, we have that $\sigma(A)=\sigma(A_{11})\cup \sigma(A_{22})$, and likewise $\sigma(A^\times)=\sigma(A_{11}^\times)\cup \sigma(A_{22}^\times)$, again by Lemma 5.9 in \cite{BGKR08}. Next observe that with respect to the decomposition $\cX=\cX_-\dot+\cX^\times_+$, we have $\Pi_r=\sbm{0&0\\0&I}$ and $I-\Pi_r=\sbm{I&0\\0&0}$. It then follows (see also \cite[Theorem 2.8]{BGKR08}) that
\begin{align}
	V_+(z) &= D_2+zD_1^{-1}\bbm{C_1&C_2}\bbm{0&0\\0&I}
	\left( I-z\bbm{A_{11}&A_{12} \\ 0&A_{22}} \right)^{-1}
	\bbm{B_1\\B_2} \notag \\
	&= D_2+zD_1^{-1}C_2(I-zA_{22})^{-1} B_2,\label{V+form2}
\end{align}
from which it becomes clear that $V_+(z)$ can be extended analytically via \eqref{V+form2} to $\rho(A)^{-1}\cup\rho(A_{22})^{-1}\cup\{0\}=\rho(A_{22})^{-1}\cup\{0\}=\C\setminus\sigma(A_{22})^{-1}$, which is an open neighborhood of $\overline{\D}$.  Using item (i) of Lemma \ref{lem:transferformulas} shows that
\begin{equation}\label{V+invform2}
	V_+(z)^{-1}=
	D_2^{-1}-z(I+D)^{-1}C_2(I-z A_{22}^\times)^{-1}B_2D_2^{-1}
\end{equation}
which extends $V_+(z)^{-1}$ analytically to the open neighborhood $\C\setminus \sigma(A_{22}^\times)^{-1}$ of $\overline{\D}$. Similarly, it follows that $V_-(z)$ and $V_-(z)^{-1}$ can be extended analytically to open neighborhoods of $\overline{\E}$ via the formulas
\begin{equation}\label{V-V-invforms2}
\begin{aligned}
V_-(z) &= D_1+zC_1(I-z A_{11})^{-1} B_1D_2^{-1}, \quad z\in \C\setminus \sigma(A_{11})^{-1},\\
V_-(z)^{-1} &= D_1^{-1}-zD_1^{-1} C_1(I-z A_{11}^\times)^{-1} B_1(I+D)^{-1}, \quad z\in \C\setminus \sigma(A_{11}^\times)^{-1}.
\end{aligned}
\end{equation}

The analytic extensions of $W_+(z)$, $W_+(z)^{-1}$, $W_-(z)$ and $W_-(z)^{-1}$ follow in a similar way. For completeness, we add the formulas of the analytic extensions. For this, note that with respect to the decomposition $\cX=\cX_-^\times \dot+ \cX_+$ the operators $A$, $A^\times$, $B$ and $C$ take the form
\begin{equation}\label{-+timesDec'}
A = \bbm{A_{11}'&0 \\ A_{21}'&A_{22}'},\ \
A^\times = \bbm{A_{11}'^\times&A_{12}'^\times\\0&A_{22}'^\times},\ \
C=\bbm{C_1'&C_2'},\ \
B=\bbm{B_1'\\B_2'}  
\end{equation}
with, using similar arguments as above, 
\begin{equation}\label{SpecDec'}
\sigma(A_{11}')\subset \E,\quad \sigma(A_{22}')\subset \D,\quad \sigma(A_{11}'^\times)\subset \E,\quad \sigma(A_{22}'^\times)\subset \D. 
\end{equation}
The formulas for $W_+(z)$, $W_+(z)^{-1}$, $W_-(z)$ and $W_-(z)^{-1}$ then turn out to be
\begin{equation}\label{Wformulas}
\begin{aligned}
W_+(z) &= D_1+zC_2'(I-z A_{22}')^{-1} B_2'D_2^{-1},\quad z\in \C\setminus \sigma(A_{22}')^{-1},\\
W_+(z)^{-1} &= D_1^{-1}-zD_1^{-1} C_2'(I-z A_{22}'^\times)^{-1} B_2'(I+D)^{-1},\quad z\in \C\setminus \sigma(A_{22}'^\times)^{-1},\\
W_-(z) &= D_2+zD_1^{-1}C_1'(I-zA_{11}')^{-1} B_1',\quad z\in \C\setminus \sigma(A_{11}')^{-1},\\
W_-(z)^{-1} &= D_2^{-1}-z(I+D)^{-1}C_1'(I-z A_{11}'^\times)^{-1}B_1'D_2^{-1},\quad z\in \C\setminus \sigma(A_{11}'^\times)^{-1}.
\end{aligned}
\end{equation}
This completes the proof. 
\end{proof}

\section*{Declarations}

\subsection*{Acknowledgements and funding}

We thank the anonymous referee for suggesting several improvements of the presentation.
%the present, much shorter and much more transparent proof of Theorem \ref{T:BRL*}.

This work is based on research supported in part by the National Research Foundation of South Africa (NRF) and the DSI-NRF Centre of Excellence in Mathematical and Statistical Sciences (CoE-MaSS). Any opinion, finding and conclusion or recommendation expressed in this material is that of the authors and the NRF and CoE-MaSS do not accept any liability in this regard. Part of the research was conducted during a visit of the second author to North-West University in March and April of 2024, supported by a scholarship of the Magnus Ehrnrooth Foundation.

\subsection*{Competing interests}

The third author is a member of the Editorial Board of Integral Equations and Operator Theory.

\subsection*{Availability of data and material} Data sharing is not applicable to this article as no datasets were generated or analysed during the current study.

\subsection*{Code availability} Not applicable.

\end{document}